\documentclass[12pt]{amsart}
\usepackage{amssymb}
\usepackage{amsmath}
\usepackage{amsthm}
\usepackage{graphicx}
\usepackage{xcolor}
\usepackage{hyperref}
\usepackage{moresize}
\usepackage{cancel}
\usepackage{bbm}
\usepackage{fullpage}
\newtheorem{theorem}{Theorem}[section]

\newtheorem{lemma}[theorem]{Lemma}
\newtheorem{remark}[theorem]{Remark}
\newtheorem{definition}{Definition}[section]
\numberwithin{equation}{section}

\newcommand{\dx}{\mathrm{d}}
\newcommand{\e}{\mathrm{e}}

\newcommand{\sumast}{\sideset{}{^*}\sum}
\newcommand{\sumtimes}{\sideset{}{^\times}\sum}

\allowdisplaybreaks

\begin{document}

\title{\small{THE AVERAGE NUMBER OF REPRESENTATIONS OF AN INTEGER AS A SUM OF THREE OR MORE PRIMES, OVER MULTIPLES OF A FIXED INTEGER}}
\author{Alessandra Migliaccio and Alessandro Zaccagnini}
\date{\today}

\subjclass[2020]{Primary 11P32. Secondary 11P55, 11P05}
\keywords{Additive problems; Goldbach-type problems}

\begin{abstract}
We extend a result by Ikeda and Suriajaya (2025) to the average number of representations of an integer as a sum of an arbitrary number $j\ge 3$ of primes. We show that the corresponding asymptotic formula holds in the general case. In the proof we use new ingredients like averages for the singular series of the generalized Goldbach problem.
\end{abstract}

\maketitle

\section{Introduction}\label{sec1}
\noindent As in our previous paper \cite{mig25}, we delve into another generalization of Ikeda and Suriajaya's results in \cite{ike25}. For a fixed $j\ge 2$, we introduce the function
\begin{equation}
\label{psi-j}
\psi_{j}(n):=\sum_{n=m_{1}+\dots+m_{j}} \Lambda(m_{1})\dots\Lambda(m_{j}).
\end{equation}
We are looking for an asymptotic estimate of
\begin{equation}
\label{def-G}
G_{q,j}(N):=\sum_{\substack{n\le N \\ q\mid n}} \psi_{j}(n),
\end{equation}
for fixed $q\ge 1$.
It is well known that $G_{1, j}(N) \sim N^j / j!$ for fixed $j \ge 2$, and more precise results have been proved for larger $j$, in particular when the Riemann Hypothesis (RH) is assumed.
For a fixed positive integer $q$ we may expect that 
\[
G_{q,j}(N) = \frac{C_{q, j}}{\phi(q)} \, G_{1,j}(N) + \mathcal{O}_{q,j}(N^{\beta}\log^\delta N),
\]
for some $\beta < j$, where $C_{q, j}$ is a constant, which is positive unless $q$ is even and $j$ is odd.
See Remark \ref{cost=0} below.

\begin{definition}
From now on, by GRH we mean that $\mathfrak{R}(\rho)=1/2$, for every non-trivial zero $\rho$ of the Dirichlet L-function $L(s,\chi)$, where $\chi$ mod $q$ is a primitive character or is induced by a primitive character $\chi_{1}$ mod $q_{1}$, s.t. $q_{1}\mid q$.   
\end{definition}

\begin{remark}
GRH will be required both for the study of the main term and the secondary term, (see \S\ref{subsecI1}-\ref{subsecI2}). 
\end{remark}

\begin{remark}
In the following, all the implicit constants in the $\mathcal{O}$ or $\mathcal{O}_{q}$ notation may depend on $j$.    
\end{remark}

\begin{theorem}
\label{th1}
Assume GRH for Dirichlet L-functions $L(s,\chi)$ associated with characters $\chi\bmod q$, for a fixed $q\ge 1$. Then, for every $j\ge 3$, we have
\[
G_{q,j}(N)=\frac{C_{q,j}}{\phi(q)}\cdot G_{1,j}(N)+\mathcal{O}_{q}(N^{j-1}\log^2 N),
\]
where 
\begin{equation}
\label{def-Cqj}
C_{q,j}:=\frac{\phi_{j-1}(q)}{\phi^{j-1}(q)}    
\end{equation}
is a positive constant, unless both $j-1$ and $q$ are even, and $\phi_{j}(q)$ is the generalized Euler function.
\end{theorem}

\begin{remark}
For the definition of the generalized Euler function we refer the reader to Tóth \cite{tot22}. Such function arises in the computation leading to \eqref{eq:Psik}.
We prove the necessary identities in Lemma \ref{sec4:1} below; see Remark \ref{rem-C} for more details. Of course, $C_{q,2}=1$.
\end{remark}

\begin{remark}
\label{cost=0}
We notice that we detect a similar phenomenon to the one pointed out in our previous paper \cite{mig25}, i.e. the constant $C_{q,j}$ may vanish according to the parity of $j$ and $q$. In particular, it is due to the relationship between $C_{q,j}$ and the generalized Euler function $\phi_{j}(q)$; see Lemma \ref{sec4:1} and Remarks \ref{phi-val} and \ref{rem-C}. This is in fact consistent with the nature of this problem.
By definition \eqref{psi-j}, if $n$ and $j$ have the same parity we may expect that $\psi_j(n) \asymp n^{j - 1}$, 
while if $n$ and $j$ have opposite parity at least one of the summands in $m_1 + \dots + m_j = n$, which gives rise to a positive product $\Lambda(m_{1})\cdots\Lambda(m_{j})$, is a power
of $2$ and $\psi_j (n) \ll \psi_{j - 1}(n) \log(n)$.
More precisely, for any $j \ge 2$ we have 
\begin{align*}
\psi_{j}(n)
&=
\sum_{n=m_{1}+\dots+m_{j}}\Lambda(m_{1})\cdots\Lambda(m_{j}) \\
&=
\sum_{m_{j}\le n}\Lambda(m_{j})\sum_{n-m_{j}=m_{1}+\dots+m_{j-1}}\Lambda(m_{1})\cdots\Lambda(m_{j-1})
=\sum_{m \le n}\Lambda(m)\cdot\psi_{j-1}(n-m).
\end{align*}
Let $q \ge 1$ be an integer.
By definition \eqref{def-G}, $G_{q, j}(N)$ is a sum restricted to multiples of $q$.
If $q$ is even and $j$ is odd, $G_{q, j}(N)$ is of lower order of magnitude than $G_{1, j}(N)$.
This is reflected in the fact that in this case, and only in this case, $C_{q, j} = 0$.
\end{remark}
\section{Lemmas and outline of the proof}
\label{sec2}
As in \cite{ike25} and in our previous work \cite{mig25}, we begin by expressing $G_{q,j}(N)$ as an integral of $I_{N}(z^{-1})$ and $F_{q,j}(z)$, where $z=\e^{-1/N}\exp(2\pi i\alpha)=:r\e(\alpha)$,
\[
F_{q,j}(z)=\sum_{n\,:\,q\mid n} \psi_{j}(n)z^{n}\quad\text{and}\quad I_{N}(z)=\sum_{n\le N}z^{n}.
\]
We obtain that
\begin{equation}
\label{sec2:1}
\begin{split}
\int_{0}^{1} F_{q,j}(z)I_{N}(z^{-1})\,\dx\alpha &= \sum_{n\,:\,q\mid n} \psi_{j}(n)\int_{0}^{1} r^{n}\e(n\alpha)\left(\sum_{n'\le N} r^{-n'}\e(-n'\alpha)\right)\,\dx\alpha \\
&=\sum_{n\,:\,q\mid n} \psi_{j}(n)\sum_{n'\le N} r^{n-n'}\int_{0}^{1}\e(\alpha(n-n'))\,\dx\alpha \\ 
&=\sum_{\substack{n\le N \\ q\mid n}} \psi_{j}(n)=G_{q,j}(N).
\end{split}
\end{equation}
In particular, for $q=1$, we have
\begin{equation}
\label{eq:G1j}
G_{1,j}(N)=\int_{0}^{1} F_{1,j}(z)I_{N}(z^{-1})\,\dx\alpha=\int_{0}^{1} \Psi^{j}(z)I_{N}(z^{-1})\,\dx\alpha,
\end{equation}
where we write
\begin{equation}
\label{def-Psi}
\Psi(z)=\sum_{n}\Lambda(n)z^n \quad\text{and}\quad \Psi(z,\chi)=\sum_{n}\chi(n)\Lambda(n)z^n,
\end{equation}
for a Dirichlet character $\chi \bmod q$.
Then, we rewrite $F_{q,j}(z)$ as a sum over characters $\chi \bmod q$. As a result, we get the equivalent version of Lemma 1 of Goldston and Suriajaya \cite{gol23}.
\begin{lemma}
\label{l1}
For $q\ge 1$, $N\ge 4$, denoting by $k:=j-1$ and defining
\begin{equation}
\label{def-Psik}
\widetilde{\Psi}_{k}(z, \chi):=\sum_{h\ge 1} \chi(h)\cdot z^h\cdot \psi_{k}(h),
\end{equation}
we have that
\[
F_{q,j}(z)=\frac{1}{\phi(q)}\sum_{\chi \bmod q}\chi(-1)\widetilde{\Psi}_{k}(z,\chi)\Psi(z,\overline{\chi}) + \mathcal{O}(N^{k}\log N\log q).
\]
\end{lemma}

\begin{proof}
    \[
F_{q,j}(z)=\sum_{\substack{n=m_{1}+\dots+m_{j} \\ q\mid n}} \Lambda(m_{1})\dots\Lambda(m_{j})\cdot z^{m_{1}+\dots+m_{j}} =: S_{1}+S_{2}
\]
say, where
\begin{align*}
&S_{1}=\sum_{m_{j}\,:\,(m_{j},q)=1} \sum_{\substack{m_{1},\dots,m_{k} \\ m_{1}+\dots+m_{k}\equiv_{q}-m_{j}}} \Lambda(m_{1})\dots\Lambda(m_{j})\cdot z^{m_{1}+\dots+m_{j}} \\
 &S_{2}=\sum_{m_{j}\,:\,(m_{j},q)>1} \sum_{\substack{m_{1},\dots,m_{k} \\ m_{1}+\dots+m_{k}\equiv_{q}-m_{j}}} \Lambda(m_{1})\dots\Lambda(m_{j})\cdot z^{m_{1}+\dots+m_{j}}.
\end{align*}
By orthogonality and definition \eqref{def-Psik}, we get
\begin{align*}
S_{1}=&\sum_{\chi\bmod q}\frac{\chi(-1)}{\phi(q)}\sum_{m_{1},\dots,m_{k}}\chi(m_{1}+\dots+m_{k})\Lambda(m_{1})\cdots z^{m_{1}+\dots+m_{k}}\sum_{m_{j}}\overline{\chi}(m_{j})\Lambda(m_{j})z^{m_{j}} \\ &= \frac{1}{\phi(q)}\sum_{\chi \bmod q}\chi(-1)\widetilde{\Psi}_{k}(z,\chi)\Psi(z,\overline{\chi}).
\end{align*}
Then, for the second summand, arguing as in \eqref{ii}, we get
\[
S_{2}\ll \sum_{m_{j}\,:\, (m_{j},q)>1}\Lambda(m_{j})|z|^{m_{j}}\biggl|\sum_{(m_{i},q)=1}\Lambda(m_{i})z^{m_{i}}\biggr|^{k}\ll N^{k}(\log N\log q),
\]
for every $i\ne j$.
\end{proof}
\begin{remark}
\label{rem-comp}
With respect to \cite{ike25}, we need the new function $\widetilde{\Psi}_{k}$ in definition \eqref{def-Psik}, since we are considering $j\ge 3$ primes, in the representation of the integer $n$. In fact, for $k=1$, we have that $\widetilde{\Psi}_{k}(z,\chi)=\Psi(z,\chi)$ and thus we come back to the problem of \cite{ike25}.
\end{remark}
\noindent Inserting this estimate into \eqref{sec2:1}, we split $G_{q,j}(N)$ into three terms that we have to evaluate, i.e. 
\[
G_{q,j}(N)=\mathcal{I}_{1}+\mathcal{I}_{2}+\mathcal{O}\left(\frac{\mathcal{I}_{3}}{\phi(q)}\right),
\]
where
\begin{equation}
\label{eq:I1}
\mathcal{I}_{1}=\frac{1}{\phi(q)}\int_{0}^{1}\widetilde{\Psi}_{k}(z,\chi_{0})\Psi(z, \overline{\chi}_{0})I_{N} (z^{-1})\,\dx\alpha
\end{equation}
\begin{equation}
\label{eq:I2}
\mathcal{I}_{2}=\frac{1}{\phi(q)}\sum_{\chi\ne\chi_{0}}\chi(-1)\int_{0}^{1}\widetilde{\Psi}_{k}(z,\chi)\Psi(z, \overline{\chi})I_{N}(z^{-1})\,\dx\alpha
\end{equation}
\begin{equation}
\label{eq:I3}
\mathcal{I}_{3}=(N^{j-1}\log N\log q) \int_{0}^{1} I_{N}(z^{-1})\,\dx\alpha.
\end{equation}

\section*{Sketch of the proof}
The estimate of $\mathcal{I}_{3}$ easily follows from (2.4)-(2.7) of \cite{ike25}, as
\begin{equation}
\label{eq:I_N}
|I_{N}(z^{-1})|\le\min\{N, \Vert \alpha \Vert^{-1}\}\implies \int_{0}^{1}|I_{N}(z^{-1})|\,\dx\alpha=\mathcal{O}(\log N),
\end{equation}
therefore\, $\mathcal{I}_{3}=\mathcal{O}(N^{j-1}\log^{2} N\log q)$. For $\mathcal{I}_{1}$, in \S\,\,\ref{subsecI1} we show that we can approximate $\widetilde{\Psi}_{k}(z,\chi_{0})$ as 
\[
\widetilde{\Psi}_{k}(z,\chi_{0})=C_{q,j}\cdot (\Psi(z)+\mathcal{O}(\log N\log q))^{j-1},
\]
where $C_{q,j}$ is defined in \eqref{def-Cqj}.
Inserting this into \eqref{eq:I1}, we have the main term of $G_{q,j}$
\[
\mathcal{I}_{1}=\frac{C_{q,j}}{\phi(q)}\cdot G_{1,j}(N)+\mathcal{O}_{q}(N^{j-1}\log^j N\log^j q),
\]
(see \S\,\ref{subsecI1} for the details).
Starting from \eqref{eq:I2},
\[
|\mathcal{I}_{2}|\ll \max_{\chi\ne\chi_{0}}\sum_{0\le t<\log_{2}N}\frac{N}{2^t}\int_{0}^{1/2h} |\widetilde{\Psi}_{k}(z,\chi)\Psi(z,\overline{\chi})|\,\dx\alpha,
\]
for $h=N/2^{t+2}$, we apply the Cauchy-Schwarz inequality to the integral and then we focus on 
\[
\int_{0}^{1/2h} |\widetilde{\Psi}_{k}(z,\chi)|^2\,\dx\alpha,
\]
since the integral of $|\Psi(z,\overline{\chi})|^2$ has already been discussed in \cite{ike25}. By Gallagher's Lemma, we reduce to the study of 
\[
\int_{0}^{y}\left|\sum_{\ell\in J}\chi(\ell)e^{-\ell/N}\cdot\psi_{k}(\ell) \right|^2\,\dx x, \qquad\text{for}\quad J=[0,x], [x,x+h] \quad\text{and}\quad y\in\{h,+\infty\},
\]
for $\psi_{k}(\ell)$ defined as in \eqref{psi-j}. Comparing to the work in \cite{ike25}, here we have to introduce Goldbach-type estimates for the $\ell^2$-norm of the error terms, for $2\le k\le 4$, while we exploit stronger results on the asymptotic behaviours by Friedlander and Goldston \cite{fri95}, when $k\ge 5$. Moreover, we will use the fact that both the generalized singular series and its square have constant average, (see definition \eqref{eq:sigmaj} and Lemmas \ref{sec4:3} and \ref{sec4:4}).

Plugging these upper bounds into the previous integral, we conclude that
 \[
 \mathcal{I}_{2}\ll_{q} N^{j-3/2}\log^2 N < N^{j-1}\log^j N \implies \mathcal{I}_{2}=\mathcal{O}_{q}(N^{j-1}\log^j N).
 \]
We refer to \S\,\ref{subsecI2} for the detailed computations.
 
 \section{Proof of Theorem \ref{th1}}\label{sec3}
As we have mentioned at the end of \S\, \ref{sec2}, we study $\mathcal{I}_{1}$, $\mathcal{I}_{2}$ and $\mathcal{I}_{3}$.  
Since the evaluation of $\mathcal{I}_{3}$ directly follows from \eqref{eq:I_N} and \eqref{eq:I3}, we focus on the study of $\mathcal{I}_{1}$ and $\mathcal{I}_2$.

\subsection{Evaluation of \texorpdfstring{$\mathcal{I}_{1}$}{I1}}
\label{subsecI1}
For $\chi=\chi_{0}$, recalling definition \eqref{def-Psik}, we have
\[
\widetilde{\Psi}_{k}(z,\chi_{0})=\sum_{(m_{1}+\dots+m_{k},\,q)=1}\Lambda(m_{1})\dots\Lambda(m_{k})\cdot z^{m_{1}+\dots+m_{k}}.
\]
In order to evaluate $\mathcal{I}_{1}$, we need some preliminary results.
\begin{lemma}
\label{l_Psi_chi0}
Let $h\in\mathbb{Z}^* _{q}$ and define the function
\begin{equation}
\label{def-psi-h}
\Psi(z,q,h):=\sum_{m\equiv_{q}h}\Lambda(m)z^{m}.
\end{equation}
Then
\[
\Psi(z,q, h)=\frac{\Psi(z,\chi_{0})}{\phi(q)} + \mathcal{O}(\log N\log q)=\frac{\Psi(z)+\mathcal{O}(\log N\log q)}{\phi(q)},
\]
where $\Psi(z,\chi_{0})=\sum_{n:(n,q)=1}\Lambda(n)z^n$, by definition \eqref{def-Psi}.
\end{lemma}

\begin{proof}
By orthogonality
\begin{align*}
\Psi(z,q, h)& = \sum_{m}\sum_{\chi \bmod q} \frac{\overline{\chi}(h)}{\phi(q)}\chi(m)\Lambda(m)\e^{-m/N}\e(m\alpha) \\ &=\frac{1}{\phi(q)}\sum_{m}\chi_{0}(m)\Lambda(m)z^{m} + \frac{1}{\phi(q)}\sum_{\chi\ne\chi_{0} \,(q)}\overline{\chi}(h)\sum_{m}\chi(m)\Lambda(m)z^{m} \\ &=\frac{\Psi(z,\chi_{0})}{\phi(q)} +\frac{1}{\phi(q)}\sum_{\chi\ne\chi_{0} \,(q)}\overline{\chi}(h)\Psi(z,\chi)=:(a)+(b). 
\end{align*}
With reference to the notation of Languasco and Zaccagnini \cite{lan12}, we denote by
\begin{align*}
V(\alpha)&=\sum_{m}\e^{-m/N}\e(m\alpha)=\frac{1}{\e^w -1}=\frac{1}{w}+\mathcal{O}(1), \quad \text{for $w=\frac{1}{N}-2\pi i\alpha$}, \\  
\delta(\chi)&=\mathbbm{1}(\chi=\chi_{0})\,,\,\, W(\chi,\alpha, V)=\sum_{m}\Lambda(m)\chi(m)\e^{-m/N}\e(m\alpha)-\delta(\chi)V(\alpha).
\end{align*}
By Lemma 2 of \cite{lan12}, we have that
\begin{align*}
(b)=\frac{1}{\phi(q)}&\sum_{\chi\bmod q}\overline{\chi}(h)\left(-\sum_{\rho}\Gamma(\rho)w^{-\rho} + E(q,N)\right), \quad\text{where}\\ \\ E(q,N)&\ll \begin{cases}
1+\log^2 q & \text{if $\chi$ is a primitive character} \\ 
(\log N)(\log q)+\log^2 q & \text{if $\chi$ is not primitive}.
\end{cases} 
\end{align*}
By Lemma 6 of \cite{lan12}, we know that
\begin{align*}
 \sum_{\chi\bmod q}\biggl|\sum_{\rho}\Gamma(\rho)w^{-\rho} \biggr|\ll \phi(q)(N^{1/2}\log(q+1) + N|\alpha|^{1/2}\log(qN)) \ll_{q} N. 
\end{align*}
Thus, for $z=\e^{-w}$, we conclude that
\[
\Psi(z,h, q)=\sum_{m\equiv_{q}h}\Lambda(m)z^{m}=(a)+(b)=\frac{\Psi(z,\chi_{0})}{\phi(q)}+\mathcal{E'},
\]
with $\mathcal{E'}$ sum of the error terms from $(b)$. By Lemma 2.2 of \cite{ike25}, we know that $\Psi(z,\chi_{0})=\Psi(z)+\mathcal{O}(\log N\log q)$. Thus, $\mathcal{E'}$ can be absorbed in the error term of $\Psi(z,\chi_{0})$ and we have the thesis.
\end{proof}

\noindent The following result gives the asymptotic behavior of $\widetilde{\Psi}_{k}(z,\chi_{0})$.
\begin{lemma}
\label{est-Psi-chi0} 
With the previous assumptions, we have
\[
\widetilde{\Psi}_{k}(z,\chi_{0})=C_{q,j}\cdot(\Psi(z)+\mathcal{O}(\phi(q)\log q\log N))^{k},
\]
where $C_{q,j}$ is defined in \eqref{def-Cqj}.
\end{lemma}
\begin{proof}
We first discuss the case $k=2$, for the sake of clarity, and then we study the asymptotic behaviour of $\widetilde{\Psi}_{k}(z,\chi_{0})$ for general $k\ge 2$. We set
\[
\Delta_{k}:=\{(m_{1},\dots,m_{k})\in(\mathbb{N}^*)^k\,:\,(m_{1}\cdots m_{k},q)>1\}
\]
and we introduce some notations:
\[
{\sumast_{m}}:=\sum_{(m,q)=1} \qquad\text{and}\qquad {\sumtimes_{m}} :=\sum_{(m,q)>1}.
\]
By definition \eqref{def-Psik}, we have
\begin{align}
\notag
 \widetilde{\Psi}_{2}(z,\chi_{0})&=\sum_{\substack{m_{1}, m_{2}: \\ (m_{1}+m_{2},\,q)=1}} \Lambda(m_{1})\Lambda(m_{2})z^{m_{1}+m_{2}}  \\
 \notag
  &=
  \biggl({\sumast_{m_{1}}} +{\sumtimes_{m_{1}}}\biggr)\Lambda(m_{1})z^{m_{1}} \cdot 
  \biggl({\sumast_{m_{2}}} +{\sumtimes_{m_{2}}}\biggr)\Lambda(m_{2})z^{m_{2}}\,\chi_{0}(m_{1}+m_{2}) \\
\notag
  &=
  {\sumast_{m_{1}, m_{2}}}\chi_{0}(m_{1}+m_{2})\Lambda(m_{1})z^{m_{1}}\Lambda(m_{2})z^{m_{2}} \\
\notag
  &\quad+
  \biggl({\sumast_{m_{1}}} {\sumtimes_{m_{2}}}+{\sumast_{m_{1}}} {\sumtimes_{m_{2}}}
  +{\sumtimes_{m_{1}}} {\sumtimes_{m_{2}}} \biggr)
  \chi_{0}(m_{1}+m_{2})\Lambda(m_{1}) \Lambda(m_{2})z^{m_1+m_{2}}  \\
\notag
  &=:
  {\sumast_{m_{1}}}{\sumast_{m_{2}}}\chi_{0}(m_{1}+m_{2})\Lambda(m_{1}) \Lambda(m_{2})z^{m_1+m_{2}} \\
\label{eq:Psi2} 
  &\quad+
  \sum_{(m_{1},m_{2})\in\Delta_{2}}\chi_{0}(m_{1}+m_{2})\Lambda(m_{1})z^{m_{1}}\Lambda(m_{2})z^{m_{2}}=:(i)+(ii),
\end{align}
say.
Arguing as for $S_{1}$ in Lemma \ref{l1}, we get that the main term is given by $(i)$, (where $(m_{1},\,q)=(m_{2},\,q)=1$). In order to split off $\chi_{0}(m_{1}+m_{2})$ from the $m_{1,2}$-sums, we introduce two more variables $h_{1},\,h_{2}\in\mathbb{Z}_{q}^*$, so that
\begin{align*}
(i)&=\sum_{h_{1}\in\,\mathbb{Z}_{q}^*}\sum_{h_{2}\in\,\mathbb{Z}_{q}^*}\chi_{0}(h_{1}+h_{2})\sum_{m_{1}\equiv_{q} h_{1}}\Lambda(m_{1})z^{m_{1}}\sum_{m_{2}\equiv_{q} h_{2}}\Lambda(m_{2})z^{m_{2}} \\ 
&=\sum_{h_{1}\in\,\mathbb{Z}_{q}^*}\sum_{h_{2}\in\,\mathbb{Z}_{q}^*}\chi_{0}(h_{1}+h_{2})\Psi(z,q,h_{1})\Psi(z,q,h_{2}),
\end{align*}
where $\Psi(z,q,h)$ is defined in \eqref{def-psi-h}.
Using Lemma \ref{l_Psi_chi0}, 
we find that, independently from the actual values of $h_{1,2}$, we have
\begin{align*}
(i)&=\sum_{h_{1}\in\,\mathbb{Z}_{q}^*}\sum_{h_{2}\in\,\mathbb{Z}_{q}^*}\chi_{0}(h_{1}+h_{2})\left(\frac{\Psi(z,\chi_{0})}{\phi(q)} + \mathcal{O}(\log N\log q) \right)^2 \\ 
&= \frac{\phi_{2}(q)}{\phi(q)^2} (\Psi(z) + \mathcal{O}(\phi(q)\log N\log q))^2 =:C_{q,3}\cdot(\Psi(z) + \mathcal{O}(\phi(q)\log N\log q))^2.
\end{align*}
\noindent For $(ii)$, we notice that 
it is enough to discuss one of the two cases $(m_{i},q)=1$ and $(m_{j},q)>1$, since they will give the highest contribution in $(ii)$. For example, if we consider $(m_{1},q)=1$ and $(m_{2},q)>1$, we get that
\begin{align}
\notag
\sum_{h_{1}\in\,\mathbb{Z}_{q}^*}\sum_{h_{2}\in\,\mathbb{Z}\setminus \mathbb{Z}_{q}^*}&\chi_{0}(h_{1}+h_{2})\sum_{m_{1}\equiv_{q} h_{1}}\Lambda(m_{1})z^{m_{1}}\sum_{m_{2}\equiv_{q} h_{2}}\Lambda(m_{2})z^{m_{2}} \\ \notag
&=\sum_{h_{1}\in\,\mathbb{Z}_{q}^*}\sum_{h_{2}\in\,\mathbb{Z}\setminus \mathbb{Z}_{q}^*}\chi_{0}(h_{1}+h_{2})\sum_{m_{1}\equiv_{q} h_{1}}\Lambda(m_{1})z^{m_{1}}{\sumtimes_{m_{2}\,(q)}}\sum_{h_{2}\equiv_{q}m_{2}}\Lambda(m_{2})z^{m_{2}} \\ \label{ii}
&= \sum_{h_{1}\in\,\mathbb{Z}_{q}^*}\sum_{h_{2}\in\,\mathbb{Z}\setminus \mathbb{Z}_{q}^*}\chi_{0}(h_{1}+h_{2})\sum_{m_{1}\equiv_{q} h_{1}}\Lambda(m_{1})z^{m_{1}}{\sumtimes_{m_{2}\,(q)}}\Lambda(m_{2})z^{m_{2}}.
\end{align}
Arguing as in Lemma 1 of \cite{gol23}, the inner sum of \eqref{ii} is 
\begin{align}
\label{est-ii}
&\ll \sum_{k}\sum_{p\mid q}\log p \cdot \e^{-p^k/N} \le \log q\cdot\biggl(\sum_{2^{k}\le N} 1 +\sum_{2^k >N}\e^{-2^{k}/N}\biggr) \le \log q\cdot \log N,  
\end{align}
which implies that \eqref{ii} $\ll |\Psi(z)|\phi(q)\log N\log q$. Similarly, if $(m_{1},\,q)>1$ and $(m_{2},q)>1$, both the inner sums of \eqref{ii} can be treated as in \eqref{est-ii}. As a result, the main part of the error term in $(ii)$ is given by $\mathcal{O}(|\Psi(z)|\phi(q)\log N\log q)$.
For every $k\ge 2$, decomposing $\widetilde{\Psi}_{k}(z,\chi_{0})$ as in \eqref{eq:Psi2}, we  are able to isolate its main term, i.e. 
\begin{align}
\notag
\widetilde{\Psi}_{k}(z,\chi_{0})&={\sum_{m_{1}}}^*\dots{\sum_{m_{k}}}^*\chi_{0}(m_{1}+\dots+m_{k})\Lambda(m_{1})z^{m_{1}}\cdots\Lambda(m_{k})z^{m_{k}} \\ \notag
&\qquad+\sum_{(m_{1},\dots,m_{k})\in\Delta_{k}}\chi_{0}(m_{1}+\dots+m_{k})(\Lambda(m_{1})z^{m_{1}}\cdots\Lambda(m_{k})z^{m_{k}}) \\ \notag
&=:\sum_{h_{1}\in\mathbb{Z}_{q}^*}\cdots \sum_{h_{k}\in\mathbb{Z}_{q}^*} \chi_{0}(h_{1}+\dots+h_{k})\sum_{m_{1}\equiv_{q}h_{1}}\Lambda(m_{1})z^{m_{1}}\cdots\sum_{m_{k}\equiv_{q}h_{k}}\Lambda(m_{k})z^{m_{k}} \\  \notag &\qquad+\sum_{(m_{1},\dots,m_{k})\in\Delta_{k}}\chi_{0}(m_{1}+\dots+m_{k})\Lambda(m_{1})z^{m_{1}}\cdots\Lambda(m_{k})z^{m_{k}}
\\ \notag
&= \sum_{\substack{h_{1},\dots,h_{k}\in\mathbb{Z}_{q}^* \\(h_{1}+\dots+h_{k},\,q)=1}} \Psi(z,q,h_{1})\cdots\Psi(z,q,h_{k}) \\ \label{eq:Psik} 
&\qquad+ \sum_{(m_{1},\dots,m_{k})\in\Delta_{k}}\chi_{0}(m_{1}+\dots+m_{k})\Lambda(m_{1})z^{m_{1}}\cdots\Lambda(m_{k})z^{m_{k}}=: (i) + (ii). 
\end{align}
As above, the leading term of each $\Psi(z,q,h_{i})$ is independent of the $h_{i}$'s: this shows why the generalized Euler function $\phi_{k}(q)$ arises; see \cite{tot22} for its definition and properties, and our Lemma \ref{sec4:1}, in particular, for the multiplicativity.
As for $k=2$, we get that
\[
(i)=\frac{\phi_{k}(q)}{\phi(q)^{k}}(\Psi(z)+\mathcal{O}(\phi(q)\log q\log N))^{k}=:C_{q,j}(\Psi(z)+\mathcal{O}(\phi(q)\log q\log N))^{k},
\]
by Lemma 2.2 of \cite{ike25}. Again, for $(ii)$, we notice that we can neglect the case $(m_{i},q)>1$, for every $i=1,\dots,k$, and that it is enough to analyse the worst-case scenario in $\Delta_{k}$, i.e. 
\begin{equation}
\label{worst}
\exists\,t\in\{1,\dots,k\}\,\,s.t.\,\, (m_{t}, q)>1 \,\wedge\, (m_{i},q)=1,\,\forall\,i\ne t.
\end{equation}
\noindent Therefore, as in \eqref{ii}, if $(m_{j},q)>1$, for every $j$, we get an error term $\ll (\log q\cdot \log N)^k$, while, in the case described in \eqref{worst}, we find that $(ii)\ll \mathcal{O}(|\Psi(z)|^{k-1}\phi(q)\log N\log q)$.
\end{proof}
\noindent Inserting \eqref{eq:Psik} into \eqref{eq:I1}, by binomial theorem we have
\begin{align*}
\mathcal{I}_{1}&=\frac{C_{q,j}}{\phi(q)}\int_{0}^{1} \Bigl(\Psi(z)+\mathcal{O}(\log N\log q)\Bigr)\Bigl(\Psi(z)+\mathcal{O}(\log q\log N)\Bigr)^{j-1} I_{N}(z^{-1})\,\dx\alpha \\ 
&=:\frac{C_{q,j}}{\phi(q)}\int_{0}^{1}(\Psi(z)+\mathcal{O}(\mathcal{E}))^jI_{N}(z^{-1})\,\dx\alpha
=\frac{C_{q,j}}{\phi(q)}\int_{0}^{1}(\Psi(z)^j + \mathcal{O}(|\Psi(z)|^{j-1}\mathcal{E}))I_{N}(z^{-1})\,\dx\alpha \\ 
&=\frac{C_{q,j}}{\phi(q)}\int_{0}^{1}\Psi(z)^j I_{N}(z^{-1})\,\dx\alpha + \mathcal{O}_{q}\left(\int_{0}^{1} |\Psi(z)|^{j-1}\mathcal{E}\,\dx\alpha\right).
\end{align*}
We recall that, by \eqref{eq:G1j}
\[
\int_{0}^{1}\Psi^j (z)I_{N}(z^{-1})\,\dx\alpha = G_{1,j}(N).
\]
For the $\mathcal{O}$-term, we estimate $|\Psi(z)|$ by partial summation and after the standard change of variable $u=t/N$, as
\begin{align*}
|\Psi(z)|\le \sum_{n}\Lambda(n) \e^{-n/N} \ll N \int_{1/N}^{+\infty}u \e^{-u}\,\dx u \ll N,
\end{align*}
which implies, together with \eqref{eq:I_N}, that
\begin{align*}
\mathcal{O}_{q}\left(\int_{0}^{1} |\Psi(z)|^{j-1}\mathcal{E}\,\dx\alpha\right)=\mathcal{O}_{q}\left(N^{j-1}\mathcal{E}\cdot\log N\right) =\mathcal{O}_{q}\left(N^{j-1}\log^2 N\log q\right). 
\end{align*}
In conclusion, we get that
\begin{align}
\label{eq:stimaI1}
\mathcal{I}_{1}=\frac{C_{q,j}}{\phi(q)}\cdot G_{1,j}(N)+\mathcal{O}_{q}(N^{j-1}\log^3 N).
\end{align}

\subsection{Evaluation of \texorpdfstring{$\mathcal{I}_{2}$}{I2}}
\label{subsecI2}
By \eqref{eq:I2}, for $\chi\ne\chi_{0}$, we have
\begin{align}
\notag
|\mathcal{I}_{2}|&\le \frac{1}{\phi(q)} \sum_{\chi\ne\chi_{0}}\int_{0}^{1} |\widetilde{\Psi}_{k}(z,\chi)\Psi(z, \overline{\chi})|\cdot |I_{N}(z^{-1})|\,\dx\alpha \\ \notag
&\le \max_{\chi\ne\chi_{0}}\int_{0}^{1} |\widetilde{\Psi}_{k}(z,\chi)\Psi(z, \overline{\chi})|\cdot |I_{N}(z^{-1})|\,\dx\alpha \\  \notag
&\ll \max_{\chi\ne\chi_{0}} \int_{0}^{1/2} |\widetilde{\Psi}_{k}(z,\chi)\Psi(z, \overline{\chi})|\min\{N, ||\alpha||^{-1}\}\,\dx\alpha \\  \notag
&\ll \max_{\chi\ne\chi_{0}} \,N \int_{0}^{1/N} |\widetilde{\Psi}_{k}(z,\chi)\Psi(z, \overline{\chi})|\,\dx\alpha  \\  \notag
&\qquad+\,\max_{\chi\ne\chi_{0}}\sum_{0\le t < \log_{2} N} \frac{N}{2^t} \int_{2^{t}/N}^{2^{t+1}/N} |\widetilde{\Psi}_{k}(z,\chi)\Psi(z, \overline{\chi})|\,\dx\alpha \\  \notag
&\le \sum_{0\le t < \log_{2} N} \frac{N}{2^t} \left(\max_{\chi\ne\chi_{0}}\int_{0}^{2^{t+1}/N} |\widetilde{\Psi}_{k}(z,\chi)\Psi(z, \overline{\chi})|\,\dx\alpha \right) \\ \label{eq:|I2|}
&\le N\sum_{0\le t<\log_{2}N} 2^{-t} \max_{\chi\ne\chi_{0}}\biggl( \int_{0}^{2^{t+1}/N} |\widetilde{\Psi}_{k}(z,\chi)|^2\,\dx\alpha\biggr)^{1/2}\biggl( \int_{0}^{2^{t+1}/N} |\Psi(z,\overline{\chi})|^2\,\dx\alpha\biggr)^{1/2}.
\end{align}
As mentioned in Remark \ref{rem-comp}, in \cite{ike25} they have $\widetilde{\Psi}_{k}=\Psi$, hence the integral in \eqref{eq:|I2|} corresponds to the $L^2$-norm of $\Psi$; in our case, instead, we need to use the Cauchy-Schwarz inequality on \eqref{eq:|I2|},
so that $|\mathcal{I}_{2}|$ can be bounded from above by the $L^2$-norm of two different functions $\widetilde{\Psi}_{k}$ and $\Psi$.
Then we apply Gallagher's Lemma (see Lemma 1.9 of Montgomery \cite{mon71}) to these two factors. By the results in $(2.23)$ of \cite{ike25}, we know the behaviour of 
\begin{equation}
\label{eq:Psi-conj}
\int_{0}^{1/2h} |\Psi(z,\overline{\chi})|^2\,\dx\alpha \ll \frac{N}{h}\log^2 N,
\end{equation}
picking
\begin{equation}
\label{h} 
h=N/2^{t+2}.
\end{equation}
Therefore, it remains to study 
\begin{equation}
\label{eq:L2-Psik}
\int_{0}^{1/2h} |\widetilde{\Psi}_{k}(z,\chi)|^2\,\dx\alpha.
\end{equation}
Before going on, it is better to summarize the strategy for the evaluation of the $L^2$-norm of $\widetilde{\Psi}_{k}(z,\chi)$.
\subsubsection*{Outline of the proof for $\mathcal{I}_{2}$}
We study the asymptotic behaviour of $\widetilde{\Psi}_{k}(z,\chi)$, which is defined in \eqref{def-Psik}: if $k> 4$, by Theorem 2 of Friedlander and Goldston \cite{fri97} we have a precise estimate for the generalized Goldbach problem, that is  
\begin{align}
\label{psik-k-big} 
\psi_{k}(\ell)=\mathfrak{S}_{k}(\ell)\frac{\ell^{k-1}}{(k-1)!}+\mathcal{O}(\ell^{k-3/2}),
\end{align}
which we can plug into \eqref{def-Psik}. Instead, for $2\le k \le 4$, we have to rely on $L^2$-estimates of
\[
\sum_{\substack{\ell\le N \\ (\ell,q)=1}} \biggl|\chi(\ell)z^\ell\cdot \biggl(\psi_{k}(\ell)-\mathfrak{S}_{k}(\ell)\frac{\ell^{k-1}}{(k-1)!}\biggr)\biggr|^2,
\]
(see  Theorem 3 of \cite{fri97}).
Thanks to our generalization of Lemma 17.4 of \cite{mon71} and Proposition of Languasco at page $95$ of \cite{lan98}, we show that both $\mathfrak{S}_{k}$ and $\mathfrak{S}_{k}^2$ have a constant average $C'=C'(k)$ and $C''=C''(k)$ respectively, (see \eqref{eq:sigmaj} for the definition of $\mathfrak{S}_{k}$ and Lemmas \ref{sec4:2} and \ref{sec4:3}). We replace such bounds for $|\widetilde{\Psi}_{k}(z,\chi)|^2$ into \eqref{eq:L2-Psik} and, by the triangle inequality, we arrive at \eqref{est-b}. Hence, we apply Gallagher's Lemma and, neglecting the constant contribution $C'$, we get from \eqref{est-L2-Psik}, that
\begin{align*}
\int_{0}^{1/2h}  |\widetilde{\Psi}_{k}(z,\chi)|^2\,\dx\alpha &\ll \frac{1}{h^2}\int_{0}^{h} \biggl|\sum_{\ell\le x}\chi(\ell)\cdot\ell^{k-1}\e^{-\ell/N} \biggr|^2\,\dx x \\ 
&\quad+ \int_{0}^{+\infty} \biggl|\sum_{x<\ell\le x+h}\chi(\ell)\ell^{k-1}\e^{-\ell/N} \biggr|^2\,\dx x \\
&=:\frac{1}{h^2}(I_{1,j}(N,h)+I_{2,j}(N,h)).
\end{align*}
In \S \, \ref{est-I1h} - \ref{est-I2h}, by partial summation and the Pólya-Vinogradov inequality, we evaluate $I_{1,j}(N,h)$ and $I_{2,j}(N,h)$ and, thanks to the substitution in \eqref{eq:|I2|}, together with \eqref{eq:Psi-conj}, we show that the main contribution of $\mathcal{I}_{2}$ is equal to $\mathcal{O}_{q}(N^{j-1}\log N)$. Sections \S\,\ref{est-a}, \S\,\ref{esti-b}, \S\,\ref{est-err}, \S\,\ref{est-errC1} and \S\,\ref{est-errC2} are devoted to show that the error terms from the approximations \eqref{sec3:1} do not contribute to the final estimate of Theorem \ref{th1}.

\begin{remark}
\label{gold}
We can see the main difference with respect to \cite{ike25} in definition \eqref{def-Psik}: in fact, for $k=1$, $\psi_{1}(\ell)$ is an empty sum, since $\widetilde{\Psi}_{1}=\Psi$, while $k\ge 2$ implies a double sum inside \eqref{def-Psik}; in particular, for $j=3$, we reduce to the standard Goldbach problem. Since the asymptotic estimate 
\begin{equation}
\label{sec3:1}
\psi_{k}(\ell):=\sum_{\ell=m_{1}+\dots+m_{k}} \Lambda(m_{1})\dots\Lambda(m_{k}) \sim \mathfrak{S}_{k}(\ell)\ell^{k-1}
\end{equation}
is known to hold only for $k>4$, we will treat separately the cases $k>4$ and $k=2,3,4$ and we will rely on $\ell^2$-norm estimates for the small values of $k$. More precisely, by Theorem 2 of \cite{fri97}, for $k>4$, by definition \eqref{def-Psik} and formula \eqref{psik-k-big}, we have
\begin{align*}
\widetilde{\Psi}_{k}(z,\chi)=\sum_{\substack{(\ell,q)=1 \\ \ell\ge 1}}\chi(\ell)\cdot z^{\ell}\left(\mathfrak{S}_{k}(\ell)\frac{\ell^{k-1}}{(k-1)!}+\mathcal{O}(\ell^{k-3/2})\right).   
\end{align*}
\end{remark}

\begin{remark}
\label{rem-err} 
For the sake of simplicity, we temporarily neglect the error terms in the estimate of $\widetilde{\Psi}_{k}$; they will be discussed in detail in \S\,\ref{est-err} - \ref{est-errC1} - \ref{est-errC2}.
\end{remark}
\noindent We resume our detailed computations.
By Lemma \ref{sec4:2}, we know that
\[
\frac{1}{N}\sum_{n\le N}\mathfrak{S}_{k}(n)=C' + \mathcal{O}\left(\frac{1}{N}\right),
\]
where $C'>0$ is a constant (computed in \eqref{eq:C0j} of Lemma \ref{sec4:2}). 
Plugging it into \eqref{eq:L2-Psik}, we reduce to the evaluation of 
\begin{align}
\notag
\int_{0}^{1/2h} |\widetilde{\Psi}_{k}(z,\chi)|^2\,\dx\alpha \notag
&\ll  \int_{0}^{1/2h} \biggl|\sum_{\substack{\ell\ge 1 \\ (\ell,q)=1}} \chi(\ell)z^\ell\cdot \biggl(\psi_{k}(\ell)-\mathfrak{S}_{k}(\ell)\frac{\ell^{k-1}}{(k-1)!}\biggr)\biggr|^2\,\dx\alpha \\ \notag
&\qquad+\int_{0}^{1/2h}\biggl| \sum_{\substack{\ell\ge 1 \\ (\ell,q)=1}} \chi(\ell)z^\ell\cdot\mathfrak{S}_{k}(\ell)\frac{\ell^{k-1}}{(k-1)!}\biggr|^2\,\dx\alpha \\ \notag
&\qquad+\int_{0}^{1/2h} \sum_{\substack{\ell\ge 1 \\ (\ell,q)=1}} |z|^{2\ell}\cdot\ell^{2(k-3/2)}\,\dx\alpha \\ \label{est-psi-tilde}
&=:(a)+(b)+(c),
\end{align}
say. For the study of $(a)$, we refer to the following section \S\,\,\ref{est-a}, where we base on $\ell^2$-estimates for $\psi_{k}(\ell)-\mathfrak{S}_{k}(\ell)\ell^{k-1}$. For the evaluation of $(b)$, we first apply Gallagher's Lemma in \S\,\,\ref{esti-b}. Then, we will require Abel's discrete partial summation formula; in this way, we will able to control the error term due the approximation of the average of $\mathfrak{S}_{k}$ by Lemma \ref{sec4:2}. As a result, we will see in \S\, \ref{est-errC1} and \S\,\ref{est-errC2} that 
\begin{align*}
(b)\ll \frac{1}{h^2}\int_{-\infty}^{+\infty}\biggl|\sum_{x\le \ell\le x+h}\chi(\ell)\cdot \e^{-\ell/N} \ell^{k-1} \biggr|^2\,\dx x.    
\end{align*}
In conclusion, $(c)$ can be treated as $I_{1,j}$ and $I_{2,j}$ in \S\,\ref{est-I1h} and \S\,\ref{est-I2h}.
\subsubsection{Evaluation of $(a)$}
\label{est-a}
In order to study $(a)$, we notice that, for $k>4$, the asymptotic behaviour in \eqref{sec3:1} is exactly computed in $(1.10)$ of Theorem $2$ of Friedlander and Goldston \cite{fri97}: it is formulated in terms of 
\[
r_{k}(\ell)=\theta_{k}(\ell):=\sum_{p_{1}+\dots+p_{k}=\ell}\log p_{1}\cdots \log p_{k},
\]
but it can still be used for our purposes, due to the asymptotic $\psi_{k}(\ell)-\theta_{k}(\ell)=\mathcal{O}(\ell^{k-1/2})$, as $\ell\to\infty$.
Thus,
\begin{align}
\notag
\sum_{\ell\ge 1}  &|\chi(\ell)\cdot z^{\ell}|^2 \left|\psi_{k}(\ell)-\frac{\mathfrak{S}_{k}(\ell)}{(k-1)!}\ell^{k-1}\right|^2\ll\sum_{\ell\ge 1}\e^{-2\ell/N}\ell^{2k-3} \ll N^{2k-2} \implies \\ \notag
&\left(\int_{0}^{1/2h} \sum_{\ell\ge 1}  |\chi(\ell)\cdot z^{\ell}|^2 \left|\psi_{k}(\ell)-\frac{\mathfrak{S}_{k}(\ell)}{(k-1)!}\ell^{k-1}\right|^2\,\dx x\right)^{1/2} \\ \label{j>4}
&\asymp \left(\int_{0}^{1/2h} \mathcal{O}(N^{2k-2})\,\dx x\right)^{1/2}
\asymp \frac{N^{k-1}}{h^{1/2}}=\frac{N^{j-2}}{h^{1/2}}.
\end{align}
Hence, plugging it into \eqref{eq:|I2|}, (after the application of the Cauchy-Schwarz inequality), together with the bound in \eqref{eq:Psi-conj}, we find that its contribution is
\begin{align}
\label{bound-a-k>4}
\ll N\sum_{0\le t<\log_{2}N} 2^{-t} \cdot N^{j-2} N^{-1/2}\cdot 2^{t/2} \cdot 2^{t/2}\log N \ll N^{j-3/2}\log^2 N,
\end{align}
thanks to our choice of $h$ in \eqref{h}.
\begin{remark}
\label{j4}
For $j=4$, $(k=3)$, we cannot apply Theorem $2$ of \cite{fri97}: in fact, the exponent $7/4$, in the error term, is too large for our purposes. In order to reduce to the $(k-3/2)=3/2$-th power, since GRH alone is not enough, we would have to introduce an additional condition, (see \S\,$3$ of \cite{fri97}). However, we will now see that GRH is enough in a mean square sense.
\end{remark}
\noindent By previous Remark \ref{j4}, for small values of $k$, (i.e. $2\le k\le 4$) we need to bound from above the previous squared $\ell^2$-norm. In particular, for $k=2$, (i.e. Goldbach's case), by Montgomery and Vaughan \cite{mon75}, given a constant $B>0$, it holds that 
\[
\sum_{\ell\le N} |\psi_{2}(\ell)-\mathfrak{S}_{2}(\ell)\ell|^2\ll_{B} N^3(\log N)^{-B}.
\]
By Gallagher's lemma, we have that
\begin{align}
\notag
(a)&\ll\frac{1}{h^2}\int_{0}^{h} \sum_{\ell\le x}\e^{-2\ell/N}\cdot |\psi_{2}(\ell)-\mathfrak{S}_{2}(\ell)\ell|^2    \,\dx x \\ \notag
&+ \frac{1}{h^2}\int_{0}^{+\infty}\sum_{x<\ell\le x+h}\e^{-2\ell/N}\cdot |\psi_{2}(\ell)-\mathfrak{S}_{2}(\ell)\ell|^2 \,\dx x \\ \label{j=3}
&=:(a_{1}) + (a_{2}).
\end{align}
For $(a_{1})$, we easily see that
\begin{align*}
(a_{1})^{1/2}\ll&\biggl(\frac{1}{h^2}\int_{0}^{h}\sum_{\ell\le x} |\chi(\ell)z^\ell\cdot (\psi_{2}(\ell)-\mathfrak{S}_{2}(\ell)\ell)|^2\, \dx x\biggr)^{1/2} \ll_{B} \frac{N^{3/2}}{h^{1/2}}
\end{align*}
and, arguing as in \eqref{bound-a-k>4}, we conclude that also this case the contribution in \eqref{eq:|I2|} is
\[
\ll N\sum_{0\le t<\log_{2}N} 2^{-t} \cdot N^{3/2} N^{-1/2}\cdot 2^{t/2} \cdot 2^{t/2}\log N \ll N^{2}\log^{2} N,
\]
by our choice of $h$ in \eqref{h}.
For $(a_{2})$, we apply the partial summation formula 
\begin{align*}
(a_{2})&=\e^{-2(x+h)/N}\sum_{\ell\le x+h} |\psi_{2}(\ell)-\mathfrak{S}_{2}(\ell)\ell|^2  - \e^{-2x/N}\sum_{\ell\le x} |\psi_{2}(\ell)-\mathfrak{S}_{2}(\ell)\ell|^2 \\
&\quad-\frac{2}{N}\int_{x}^{x+h}\sum_{\ell\le t} |\psi_{2}(\ell)-\mathfrak{S}_{2}(\ell)\ell|^2 \e^{-2t/N}\,\dx t \\ 
&\ll_{k} hx^2\e^{-2x/N}.
\end{align*}
As before, we conclude that
\begin{align*}
(a_{2})^{1/2}\ll \biggl(\frac{1}{h^2}\int_{0}^{+\infty}hx^2\e^{-2x/N}\,\dx x \biggr)^{1/2} \ll \frac{N^{3/2}}{h^{1/2}}, 
\end{align*}
which implies that the contribution of $(a_{2})$ to \eqref{eq:|I2|} is
\begin{align*}
\ll N\sum_{0\le t<\log_{2}N} 2^{-t} N \cdot 2^{t/2}\cdot 2^{t/2}\log N \ll N^2 \log^2 N.
\end{align*}
Instead, for $k=3,4$, we refer to Theorem 3 of \cite{fri97}, that we recall below, and then we argue as in \eqref{j=3} and \eqref{bound-a-k>4}.
\begin{theorem}[Friedlander and Goldston \cite{fri97}]
\label{th2}
Assume GRH. Then,
\begin{align*}
&\sum_{\ell\le N}\left|r_{3}(\ell)-\frac{1}{2}\mathfrak{S}_{3}(\ell)\ell^2 \right|^{2} \ll N^{4}\log^{7}N, \\  
&\sum_{\ell\le N}\left|r_{4}(\ell)-\frac{1}{6}\mathfrak{S}_{4}(\ell)\ell^3\right|^{2} \ll N^{6}\log^4 N.
\end{align*}
\end{theorem}

\begin{remark}
\label{nguyen}    
The recent work of Nguyen \cite{ngu26} gives an asymptotic estimate for $G_{k}(N)$, for every $k\ge 2$; however, here we are looking for punctual estimates of $\psi_{k}(n)$ or for $\ell^2-$norm upper bounds of $|\psi_{k}(\ell)-\mathfrak{S}_{k}(\ell)\cdot\ell^{k-1}|$. 
\end{remark}

\subsubsection{Evaluation of $(b)$}
\label{esti-b}
By Gallagher's Lemma and adding and subtracting the average $C'=2$ of the generalized singular series, (see Lemma \ref{sec4:2} for the computation of $C'$), we get
\begin{align}
\notag
(b)&\ll_{j} \frac{1}{h^2}\int_{-\infty}^{+\infty}\biggl|\sum_{x<\ell\le x+h} \chi(\ell)e^{-\ell/N}\ell^{j-2}(\mathfrak{S}_{j-1}(\ell)-C') \biggr|^2\,\dx x \\  \label{est-b}
&\qquad+ \frac{1}{h^2}\int_{-\infty}^{+\infty}\biggl|\sum_{x<\ell\le x+h} \chi(\ell)e^{-\ell/N}\ell^{j-2} \biggr|^2\,\dx x.
\end{align}
Since we will find an upper-bound of the first summand of \eqref{est-b} in \S\,\,\ref{est-err}, now we focus on the second one. We have
\begin{align}
\notag
\frac{1}{h^2}\int_{-\infty}^{+\infty}\biggl|\sum_{x<\ell\le x+h} \chi(\ell)e^{-\ell/N}\ell^{j-2} \biggr|^2\,\dx x&\ll \frac{1}{h^2}\int_{-\infty}^{+\infty} \left|\sum_{x<\ell\le x+h}\chi(\ell)\cdot \ell^{j-2}\e^{-\ell/N} \right|^2\,\dx x \\ \notag
&=\frac{1}{h^2}\int_{0}^{h} \left|\sum_{\ell\le x}\chi(\ell)\cdot\ell^{j-2}\e^{-\ell/N} \right|^2\,\dx x \\ \notag
&\quad+ \frac{1}{h^2}\int_{0}^{+\infty} \left|\sum_{x<\ell\le x+h}\chi(\ell)\ell^{j-2}\e^{-\ell/N} \right|^2\,\dx x \\ \label{est-L2-Psik}
&=:\frac{1}{h^2}(I_{1,j}(N,h)+I_{2,j}(N,h)).
\end{align}

\subsubsection{Estimate of \texorpdfstring{$I_{1,j}(N,h)$}{I(1,j)}}
\label{est-I1h}
By partial summation
\begin{align*}
\sum_{\ell\le x}\chi(\ell)\cdot\ell^{j-2}\e^{-\ell/N} &= \biggl(\sum_{\ell\le x}\chi(\ell) \biggr)x^{j-2}\e^{-x/N}-\int_{1}^{x}\biggl(\sum_{\ell\le t}\chi(\ell) \biggr)\frac{d}{dt}(t^{j-2}\e^{-t/N})\,\dx t \\ 
&= \biggl(\sum_{\ell\le x}\chi(\ell) \biggr)x^{j-2}\e^{-x/N}+\frac{1}{N}\int_{1}^{x} \biggl(\sum_{\ell\le t}\chi(\ell) \biggr)t^{j-2}\e^{-t/N}\,\dx t \\ 
&\qquad - (j-2)\int_{1}^{x} \biggl(\sum_{\ell\le t}\chi(\ell) \biggr)t^{j-3}\e^{-t/N}\,\dx t.
\end{align*}
By Polya-Vinogradov's inequality and the explicit computation of the integral, we notice that
\begin{align*}
\frac{1}{N}\int_{1}^{x} \biggl(\sum_{\ell\le t}\chi(\ell) \biggr)t^{j-2}\e^{-t/N}\,\dx t \ll_{q} \frac{1}{N}\int_{1}^{x} t^{j-2}\e^{-t/N}\,\dx t \ll_{q,j}x^{j-2}\e^{-x/N}
\end{align*}
and that
\begin{align*}
 \int_{1}^{x} \biggl(\sum_{\ell\le t}\chi(\ell) \biggr)t^{j-3}\e^{-t/N}\,\dx t \ll_{q}  \int_{1}^{x}t^{j-3}\e^{-t/N}\,\dx t\ll_{q,j} Nx^{j-3}.
\end{align*}
Thus, we can bound the square of the previous sum 
\begin{align*}
 \biggl|\sum_{\ell\le x}\chi(\ell)\cdot\ell^{j-2}\e^{-\ell/N}\biggr|^2 \ll_{q,j} N^2  x^{2(j-3)}.
\end{align*}

Integrating between $0$ and $h$ and dividing by $h^2$, we conclude that
\begin{align*}
\frac{I_{1,j}(N,h)}{h^2}\ll_{j,q} \frac{N^2}{h^2}h^{2j-5}=\frac{N^{2j-5}}{(2^{t+2})^{2j-7}}\ll \frac{N^{2j-5}}{2^{t(j-7)}},
\end{align*}
 by our choice of $h$ in \eqref{h}. Thus, also recalling \eqref{eq:Psi-conj}, the contribution of $I_{1,j}(N,h)$ to \eqref{eq:|I2|} is 
 \begin{align*}
&\ll_{q,j}N\sum_{0\le t< \log_{2}N} 2^{-t} N^{j-5/2} (2^{-t})^{j-7/2} \cdot 2^{t/2}\log N\asymp_{q,j} N^{j-3/2}\sum_{0\le t< \log_{2}N} 2^{-t(j-3)}\\ &\ll_{q,j}N^{j-3/2}\log^2 N < N^{j-1}\log^2 N. 
 \end{align*}
\subsubsection{Estimate of \texorpdfstring{$I_{2,j}(N,h)$}{I(2,j)}}
\label{est-I2h}
By partial summation,
\begin{align*}
&\sum_{x<\ell\le x+h} \chi(\ell)\ell^{j-2}\e^{-\ell/N} = \biggl(\sum_{\ell\le x+h} \chi(\ell)\biggr)(x+h)^{j-2}\e^{-(x+h)/N}-\biggl(\sum_{\ell\le x} \chi(\ell)\biggr)x^{j-2}\e^{-x/N}  \\ 
&\qquad+ \int_{x}^{x+h}\biggl(\sum_{\ell\le t}\chi(\ell)\biggr)\frac{d}{dt}(t^{j-2}\e^{-t/N})\,\dx t =:(A)+(B).
\end{align*}
For $(B)$, arguing as in \S\, \ref{est-I1h}, we have
\begin{align*}
(B)&\ll_{q,j} \frac{1}{N}\int_{x}^{x+h} \e^{-t/N}t^{(j-2)}\,\dx t -\int_{x}^{x+h} \e^{-t/N}t^{j-3}\,\dx t.
\end{align*}
Now, by integration by parts, the first integral is 
\begin{align*}
&\ll_{j,q} 
[(x+h)^{(j-2)}\e^{-(x+h)/N}-x^{(j-2)}\e^{-x/N}]\ll_{q,j}hx^{j-3}\e^{-x/N},
\end{align*}
while the second one is
\begin{align*}
&\ll_{q,j} N\cdot [(x+h)^{(j-3)}\e^{-(x+h)/N}-x^{(j-3)}\e^{-x/N}]\ll_{q,j}Nhx^{j-4}e^{-x/N}. 
\end{align*}
Since $x\to +\infty$, $hx^{j-3}\ge Nhx^{j-4}$, we get that
\begin{align*}
\frac{1}{h^2}\int_{0}^{+\infty} (hx^{j-3}e^{-x/N})^2\,\dx x \asymp_{q,j}N^{2j-5}\int_{0}^{+\infty}u^{2(j-3)}e^{-u}\,\dx u \ll_{q,j} N^{2j-5},  
\end{align*}
thanks to the usual change of variable $u=2x/N$ and our choice of $h$ in \eqref{h}.
In order to bound $(A)$, we refer again to Pólya–Vinogradov inequality and we find that
\begin{align*}
(A)^2&\ll_{q,j} [(x+h)^{2(j-2)}\e^{-2(x+h)/N}+x^{2(j-2)}\e^{-2x/N}].
\end{align*}
Thus, as for $(B)$, we get that
\[
\frac{1}{h^2}\int_{0}^{+\infty}(A)^2\,\dx x \ll_{j,q} \frac{N^{2j-3}}{h^2}=N^{2j-5}\cdot 2^{2(t+2)}.
\]
As a result, summing the estimates of $(A)$ and $(B)$, we have that
\begin{align*}
\frac{I_{2,j}(N,h)}{h^2}&\le\frac{1}{h^2}\int_{0}^{+\infty} [(A)^2+(B)^2]\,\dx x \ll_{j,q} N^{2j-5}(1 + 2^{2(t+2)})\ll N^{2j-5}\cdot 2^{2t}.
\end{align*}
Summing over $t$, as in \cite{ike25}, we conclude that the entire contribution of $I_{2,j}(N,h)$ to \eqref{eq:|I2|} is
\begin{align}
\notag
&\ll_{j,q} N\sum_{0\le t< \log_{2}N} \frac{2^{(3/2)t}}{2^{t}} (N^{j-5/2}\cdot \log N) \ll_{j,q} N^{j-3/2}\log N \sum_{0\le t< \log_{2}N}2^{t/2} \\ 
&< N^{j-1}\log N.
\label{I2-fin}
\end{align}

\subsubsection{Estimates of the error terms in \texorpdfstring{$\mathcal{I}_{2}$}{I2}}
\label{est-err}
In conclusion, we compute the contribution of the remaining error terms in $\widetilde{\Psi}_{j-1}(z,\chi)$, i.e. we focus on the first summand in \eqref{est-b} of \S\,\ref{esti-b}. In what follows, the asymptotic (quadratic) average of $\mathfrak{S}_{j-1}$ will be required.
Recalling definition \eqref{def-Psik} and the estimate in \eqref{est-psi-tilde}, by Gallagher's Lemma and by the triangle inequality, adding and subtracting the average of the singular series, we have the inequality in \eqref{est-b}. In order to have the thesis, we have to find an upper bound for
\begin{align}
\notag
\frac{1}{h^2}&\int_{-\infty}^{+\infty}\biggl|\sum_{x<\ell\le x+h} \chi(\ell)e^{-\ell/N}\ell^{j-2}(\mathfrak{S}_{j-1}(\ell)-C') \biggr|^2\,\dx x \\ \notag
&=\frac{1}{h^2}\int_{0}^{h}\biggl|\sum_{\ell\le x} \chi(\ell)e^{-\ell/N}\ell^{j-2}(\mathfrak{S}_{j-1}(\ell)-C') \biggr|^2\,\dx x \qquad \\  \notag
&\qquad+\frac{1}{h^2}\int_{0}^{+\infty}\biggl|\sum_{x<\ell\le x+h} \chi(\ell)e^{-\ell/N}\ell^{j-2}(\mathfrak{S}_{j-1}(\ell)-C') \biggr|^2\,\dx x \\ \label{C-D} 
&=:(C_{1})+(C_{2}).
\end{align}

\subsubsection{Evaluation of \texorpdfstring{$C_{1}$}{C1}} 
\label{est-errC1}
By partial summation, with $A(x)=\sum_{\ell\le x}\chi(\ell)(\mathfrak{S}_{j-1}(\ell)-C')$ and $f(\ell)=\ell^{j-2}\e^{-\ell/N}$, we get
\begin{align}
\notag
\sum_{\ell\le x}\chi(\ell)(&\mathfrak{S}_{j-1}(\ell)-C')\ell^{j-2}\e^{-\ell/N} = A(x)f(x)-\int_{1}^{x}A(t)f'(t)\,\dx t \\ \label{C-1-i}\quad\implies 
&\sum_{\ell\le x}\chi(\ell)(\mathfrak{S}_{j-1}(\ell)-C')f(\ell) =A(x) f(x) \\  \label{C-1-ii}
&\qquad+\int_{1}^{x}A(t) t^{(j-2)}\e^{-t/N}((j-2)t^{-1} - N^{-1})\,\dx t. 
\end{align}
To estimate $A(x)$, we use the discrete Abel partial summation formula
\[
\sum_{\ell=1}^{\lfloor x \rfloor} u_{\ell}v_{\ell}=u_{\lfloor x \rfloor+1}\sum_{\ell=1}^{\lfloor x \rfloor} v_{\ell} - \sum_{\ell=1}^{\lfloor x \rfloor} (u_{\ell+1}-u_{\ell})\sum_{s=1}^\ell v_{s},
\]
with $u_{\ell}=(\mathfrak{S}_{j-1}(\ell)-C')$ and $v_{\ell}=\chi(\ell)$. In this way, we have
\begin{align}
\notag
\sum_{\ell\le x}\chi(\ell)(\mathfrak{S}_{j-1}(\ell)-C') &= (\mathfrak{S}_{j-1}(\lfloor x \rfloor+1)-C') \cdot \sum_{\ell=1}^{\lfloor x \rfloor}\chi(\ell) \\ \label{Abel-S}
&\qquad+\sum_{\ell=1}^{\lfloor x \rfloor} (\mathfrak{S}_{j-1}(\ell+1)-\mathfrak{S}_{j-1}(\ell))\sum_{s=1}^{\ell}\chi(s) .
\end{align}

By definition \eqref{eq:sigmaj} and standard estimates, we have that $\mathfrak{S}_{j-1}(\ell) \ll 1$.
Then we highlight the following property of the singular series, which can be found out by direct computation:
\[
\mathfrak{S}_{j-1}(\ell)\ne 0 \iff j-1\equiv \ell \mod{2}.
\]
Since $\ell$ is even if and only if $\ell+1$ is odd (and viceversa), we can simplify our estimates, as
\begin{align*}
(\mathfrak{S}_{j-1}(\ell+1)-\mathfrak{S}_{j-1}&(\ell))=\mathfrak{S}_{j-1}(\ell) 
\\  \Downarrow \\ 
\sum_{\ell\le \lfloor x \rfloor}(\mathfrak{S}_{j-1}(\ell+1)-\mathfrak{S}_{j-1}(\ell)) &=\sum_{\ell\le \lfloor x \rfloor}\mathfrak{S}_{j-1}(\ell) \ll_{q,j} x,
\end{align*}
where the last inequality follows from Lemma \ref{sec4:2}. 
Therefore, by previous argument and recalling the Pólya–Vinogradov inequality, we have the following bound for \eqref{Abel-S}
{\small
\begin{equation}
\label{sec3:3}
\sum_{\ell\le x}\chi(\ell)(\mathfrak{S}_{j-1}(\ell)-C') \ll_{q,j} x.
\end{equation}
}
Furthermore, by Lemma \ref{sec4:3}, we also know that 
\begin{align*}  
 \sum_{\ell\le \lfloor x\rfloor }\mathfrak{S}_{j-1}(\ell)^2 \ll_{j} x,
\end{align*}
which implies that 
\begin{align}
 \label{est-sq-s} 
 \sum_{\ell\le \lfloor x\rfloor }|\chi(\ell)(\mathfrak{S}_{j-1}(\ell)-C')|^2\ll_{j}x.
\end{align}
\newline As a consequence, \eqref{C-1-i} and \eqref{C-1-ii} are bounded by
\begin{align*}
&\ll_{q} x^{j-1}\e^{-x/N}+\frac{1}{N}\int_{1}^{x} t\cdot t^{j-2}\e^{-t/N}\,\dx t +\int_{1}^{x} t\cdot t^{j-3}\e^{-t/N}\,\dx t.
\end{align*}
Arguing as in \S\,\ref{est-I1h}, we have that
\begin{align*}
\frac{1}{N}&\int_{1}^{x} t\cdot t^{j-2}\e^{-t/N}\,\dx t\ll x^{j-1}\e ^{-x/N},  \\
&\int_{1}^{x}t\cdot t^{j-3}\e^{-t/N}\,\dx t \ll N x^{j-2}e^{-x/N}.
\end{align*}
Recalling that, in this range for $x$, we have $Nx^{j-2}\ge x^{j-1}$, integrating by parts between $0$ and $h$ and dividing by $h^2$, we get
\begin{align*}
(C_{1}) &\ll_{q,j} \frac{N^2}{h^2}\int_{0}^{h} x^{2(j-2)}\e^{-2x/N}\,\dx x \\  &\ll_{q,j}\frac{1}{h^2} N^2 h^{2j-4}\cdot \e^{-2h/N}\\ 
&\ll_{q,j} N^2 \cdot h^{2j-6}\asymp \frac{N^{2j-4}}{(2^{t+2})^{2j-6}}.
\end{align*}
As a result, inserting in \eqref{I2-fin}, we get an error term
\begin{align}
\notag
&\ll_{q,j} N\sum_{0\le t<\log_{2}N} 2^{-t} N^{j-2}\cdot 2^{-(t+2)(j-3)} \cdot 2^{t/2\,+1}\log N \\ \label{est-C1} &\ll N^{j-1}\log N\sum_{0\le t<\log_{2}N}  2^{-t(j-5/2)}\ll N^{j-1}\log^2 N.  
\end{align}

\subsubsection{Evaluation of \texorpdfstring{$C_{2}$}{C2}}
\label{est-errC2}
By partial summation and Cauchy-Schwarz's inequality, with $A(x)=\sum_{\ell\le x}\chi(\ell)(\mathfrak{S}_{j-1}(\ell)-C')$ and $f(x)=x^{j-2}\e^{-x/N}$,
\begin{align*}
\biggl|\sum_{x<\ell\le x+h} \chi(\ell)\e^{-\ell/N}\ell^{j-2}(\mathfrak{S}_{j-1}(\ell)-C') \biggr|^2 &\le 2\left| A(x+h)(x+h)^{j-2}\e^{-(x+h)/N} - A(x)x^{j-2}\e^{-x/N}\right|^2 \\ 
&\quad+2\biggl| \int_{x}^{x+h} A(t) f'(t)\,\dx t\biggr|^2 \\ 
&\ll\left| A(x+h)(x+h)^{j-2}\e^{-(x+h)/N} - A(x)x^{j-2}\e^{-x/N}\right|^2  \\ 
&\quad+\frac{h}{N^2}\int_{x}^{x+h} |A(t)|^2 t^{2(j-2)}\e^{-2t/N}\,\dx t \\
&=: (C_{2,1})+ (C_{2,2}).
\end{align*}
Basing on \eqref{est-sq-s} and by the mean value theorem, we have
\begin{align*}
(C_{2,1})&\ll_{q,j} \frac{h}{N^2}\int_{x}^{x+h} t\cdot t^{2(j-2)}\e^{-2t/N}\,\dx t \\   &\ll_{q,j} N^{-2}h^2\cdot (\xi')^{2j-3}\e^{-2\xi'/N},
\end{align*}
for some $\xi'\in (x,x+h)$.
Therefore, 
\begin{align}
\notag
\frac{1}{h^2}\int_{0}^{+\infty}(C_{2,1})\,\dx x&\ll_{q,j}\frac{1}{h^2}\int_{0}^{+\infty} N^{-2}h^2\cdot (\xi')^{2j-3}\e^{-2\xi'/N}\,\dx x \\ \label{est-C21} &\ll N^{2j-4}, 
\end{align}
by the usual change of variable $u=2\xi'/N$. Thus, the contribution of $(C_{2,1})$ is negligible as well, since
\begin{align*}
N\sum_{0\le t<\log_{2}N} 2^{-t} N^{j-2} \cdot 2^{t/2} \log N < N^{j-1}\log^2 N. 
\end{align*}
For $(C_{2,2})$, we first add and subtract $A(x+h)(x+h)^{j-2}\e^{-x/N}$ and $A(x)(x+h)^{j-2}\e^{-x/N}$ to the absolute value and we repeatedly apply the inequality $|a+b|^2 \le 2|a|^2 + 2|b|^2$:
\begin{align}
\notag
(C_{2,2})&\le 2|A(x+h)|^2(x+h)^{2(j-2)}|\e^{-(x+h)/N}-\e^{-x/N}|^2 \\  \notag
&\quad+ 4(x+h)^{2(j-2)}\e^{-2x/N} |A(x+h)-A(x)|^2 \\ \label{C22}
&\quad+4|A(x)|^2\e^{-2x/N} |(x+h)^{j-2} - x^{j-2}|^2.
\end{align}
By \eqref{sec3:3}, we know that $|A(x)|^2 \ll x$, $|A(x+h)|^2 \ll x+h$. Using these upper bounds, we obtain that 
\begin{align*}
|A(x+h)-A(x)|^2&\le3\biggl|[\mathfrak{S}_{j-1}(\lfloor x+h \rfloor+1)-C']\sum_{\ell=\lfloor x \rfloor+1}^{\lfloor x+h \rfloor}\chi(\ell)\biggr|^2 \\ 
&\quad+3\biggl|[\mathfrak{S}_{j-1}(\lfloor x+h \rfloor +1)-\mathfrak{S}_{j-1}(\lfloor x \rfloor+1)]\sum_{\ell\le \lfloor x \rfloor}\chi(\ell)\biggl|^2  \\ 
&\quad + 3\biggl|\sum_{\ell=\lfloor x \rfloor+1}^{\lfloor x+h \rfloor}\mathfrak{S}_{j-1}(\ell)\sum_{s\le \ell} \chi(s)\biggr|^2 \ll_{q,j} h.
\end{align*}
Again by the mean value theorem, we estimate 
\begin{align*}
|(x+h)^{j-2} - x^{j-2}|^2 \ll h^2\xi^{2(j-3)}, 
\end{align*}
for some $\xi\in (x,x+h)$.
Moreover, it trivially holds that $|\e^{-(x+h)/N}-\e^{-x/N}|^2 \le e^{-2x/N}$.
Rewriting \eqref{C22}, integrating between $0$ and $\infty$ and multiplying by $h^{-2}$, we get that
\begin{align*}
\frac{1}{h^{2}}\int_{0}^{+\infty}(C_{2,2})&\,\dx x \ll_{q,j}\frac{1}{h^{2}}\int_{0}^{+\infty}(x+h)^{2j-3}\e^{-2x/N}\,\dx x \\ 
&\quad\qquad+ \frac{1}{h^{2}}\int_{0}^{+\infty}h(x+h)^{2j-4}\e^{-2x/N}\,\dx x \\ 
&\quad\qquad+\frac{1}{h^2}\int_{0}^{+\infty}h^2 \xi^{2j-6} x \e^{-2x/N}\,\dx x.
\end{align*}
By the change of variable $u=x+h$, we have
\begin{align}
\notag
\frac{1}{h^{2}}\int_{0}^{+\infty}(x+h)^{2j-3}\e^{-2x/N}\,\dx x&\ll \frac{1}{h^{2}}\int_{h}^{+\infty}u^{2j-3}\e^{-2(u-h)/N}\,\dx u \\ \label{C22-i}
&\ll \frac{1}{h^2}Nh^{2j-3} \ll N^{2j-4}\cdot 2^{-t(2j-5)}
\end{align}
and similarly
\begin{align}
\label{C22-ii}
\frac{1}{h^{2}}\int_{0}^{+\infty}h(x+h)^{2j-4}\e^{-2x/N}\,\dx x \ll  \frac{N}{h} h^{2j-4}=N^{2j-4}\cdot 2^{-t(2j-5)}. 
\end{align}
Since $\xi\in (x,x+h)$, $\xi\sim x$. Thus
\begin{align}
\label{C22-iii}
\frac{1}{h^2}\int_{0}^{+\infty}h^2 \xi^{2j-6} x \e^{-2x/N}\,\dx x\ll \int_{0}^{+\infty} x^{2j-5}\e^{-2x/N}\,\dx x \ll N^{2j-4}.  
\end{align}
Collecting \eqref{C22-i}, \eqref{C22-ii} and \eqref{C22-iii}, we have that 
\begin{align}
\notag
\frac{1}{h^2}\int_{0}^{+\infty}(C_{2,2})\,\dx x \ll_{q,j} N^{2j-4}\cdot 2^{-t(2j-5)}+ N^{2j-4},
\end{align}
which implies that the error term due to $(C_{2,2})$ is
\begin{align*}
&\ll_{q,j} N\sum_{0\le t<\log_{2}N} 2^{-t} N^{j-2}\cdot 2^{t/2}\log N \\ &\ll N^{j-1} \log N \sum_{0\le t<\log_{2}N}  2^{-t/2} \ll N^{j-1}\log N.
\end{align*}
In conclusion, the approximation $\sum_{\ell\le N} \mathfrak{S}_{k}(\ell)\sim_{j} N$, does not change the order $N^{j-1}$ of the expected error term. 

\section{The Generalized singular series}

\begin{lemma}
\label{sec4:1}
    Let $q\ge 1$ and let $k\ge 2$. Denote by
\[
\phi_{k}(q):=\sum_{h_{1}\in\,\mathbb{Z}^* _{q}}\dots\sum_{h_{k}\in\,\mathbb{Z}^* _{q}}\chi_{0}(h_{1}+\dots+h_{k}).
\]
Then, $\phi_{k}(q)$ is a multiplicative function, given by
\[
\phi_{k}(q)=q^{k}\prod_{p\mid q}\left(1-\frac{1}{p} \right)\left(\left(1-\frac{1}{p} \right)^{k}-\frac{(-1)^{k}}{p^{k}}\right).
\]
\end{lemma}

\begin{proof}
It follows from Theorem 2.1 of \cite{tot22}, noting that $\phi_{k}(q)$ is equal to the generalized Euler function

\[
\phi_{k}(q):=\sum_{\substack{h_{1},\dots,h_{k}\in\,\mathbb{Z}_{q}^* \\ (h_{1}+\dots+h_{k},\,q)=1}} 1,
\]
as $m_{i}\equiv_{q}\,h_{i}$, $(m_{i}, \,q)=1$ $\implies$ $(h_{i},\,q)=1$, $\forall\,i$. 
\end{proof}

\begin{remark}
\label{rem-C}
By direct computation of $C_{q,j}:=\phi_{k}(q)/\phi^k (q)$, we find that 
\begin{equation}
\label{eq:C}
C_{q,j}=\prod_{p\mid q}\frac{p-1}{p}\cdot \begin{cases}
1-(p-1)^{-j+1} & \text{if $j\equiv_{2}\,1$} \\ 
1+(p-1)^{-j+1} &\text{if $j\equiv_{2}\,0$}.
\end{cases}
\end{equation}
\end{remark}
\begin{remark}
\label{phi-val}
By Lemma \ref{sec4:1}, we can verify that $\phi_{k}(q)>0$, unless both $k$ and $q$ are even. In fact, 
\[
\left(1-\frac{1}{p} \right)^{k}-\frac{(-1)^{k}}{p^{k}}=0 \implies 2\mid k\,\wedge \, p=2 \implies 2\mid q.
\]
Such property confirms our previous observations in Remark \ref{cost=0}.
\end{remark}
\begin{definition}
The generalized Goldbach singular series is given by
\begin{align}
\label{eq:sigmaj}
\mathfrak{S}_{j}(\ell)&=\prod_{p\mid \ell} \left(1-\frac{(-1)^{j-1}}{(p-1)^{j-1}} \right)\prod_{p\nmid \ell}\left(1-\frac{(-1)^{j}}{(p-1)^j} \right) \\ \notag
&=\begin{cases}
\mathfrak{S}^{(j)}\prod_{p\mid \ell}\left(1+\dfrac{(-1)^{j}p}{(p-1)^{j}-(-1)^j}\right) & \text{if $j\equiv_{2} \ell$} \\
0 & \text{if $j\not\equiv_{2} \ell$},
\end{cases}
\end{align}
where
\begin{equation}
\label{eq:C0j}
\mathfrak{S}^{(j)}=2\prod_{p\ne 2}\left(1+\frac{(-1)^{j+1}}{(p-1)^j} \right).
\end{equation}
\end{definition}

\noindent Now we generalize formula (17.11) of Lemma 17.4 of \cite{mon71}.
\begin{lemma}
    \label{sec4:2}
    For every $j\ge 3$, the average of the generalized singular series $\mathfrak{S}_{j}(\ell)$ is equal to
\[
\frac{1}{N}\sum_{\ell\le N}\mathfrak{S}_{j}(\ell)=2+\mathcal{O}_{j}\left(\frac{\log N}{N}\right).
\]
\end{lemma}

\begin{proof}
Since, by $(2)$ of \cite{lan12}, we can express $\mathfrak{S}_{j}$ as in \eqref{eq:sigmaj},
we just focus on the case $j\equiv_{2} \ell$. Mimicking the proof of Lemma 17.4 of \cite{mon71}, for $r=1$, we introduce the function
\begin{equation}
\label{f_r}
f_{1}(\ell)=\begin{cases}
\mu^j (\ell)\prod_{p\mid \ell} \dfrac{(-1)^{j}p}{(p-1)^j - (-1)^j} & \text{if $(\ell, 2)=1$} \\
0 & \text{otherwise.}
\end{cases}
\end{equation}
Recalling Definition \eqref{eq:sigmaj} when $j\equiv \ell\bmod 2$, the goal is to prove that
\begin{align*}
\sum_{d\mid \ell}f_{1}(d) =\frac{\mathfrak{S}_{j}(\ell)}{\mathfrak{S}^{(j)}}.  
\end{align*}
In fact, in this way, we will get that
\begin{align*}
\sum_{\ell\le N} \sum_{d\mid \ell}f_{1}(d) =\sum_{\ell\le N}\frac{\mathfrak{S}_{j}(\ell)}{\mathfrak{S}^{(j)}},
\end{align*}
summations which will give us the desired estimate, via the function $f_{1}$. 
Exchanging the order of summation, we get
\begin{align*}
\sum_{\ell\le N}\sum_{d\mid \ell} f_{1}(d)&=\sum_{d\le N}f_{1}(d)\biggl[\frac{N}{d}\biggr]=N\sum_{d\le N}\frac{f_{1}(d)}{d} + \mathcal{O}\left(\sum_{d\le N}f_{1}(d)\right) \\ 
&=N\sum_{d\ge  1}\frac{f_{1}(d)}{d} + \mathcal{O}\left(N\sum_{d> N}\frac{f_{1}(d)}{d}\right) + \mathcal{O}\left(\sum_{d\le N}f_{1} (d)\right) \\  
&=N\sum_{d\ge 1}\frac{f_{1}(d)}{d}+\mathcal{O}(\log N).
\end{align*}
Now we notice that
\[
\sum_{d\mid \ell}f_{1}(d)=\prod_{i=0}^k \sum_{m=0}^{\alpha_{i}}f_{1}(p_{i}^m)=\prod_{i=0}^k(f_{1}(1)+f_{1}(p_{i})),
\]
since
\begin{align*}
f_{1}(p_{i}^m)&\ne 0 \iff p\mid p_{i}^m\, \wedge\, (p_{i}^m, 2)=1 \iff p_{i}=p\ne 2 \\    
&\text{and}\quad \mu(p_{i}^m)^j\ne 0 \iff m\in\{0,1\}.
\end{align*}
Thus, 
\[
\sum_{d\mid \ell}f_{1}(d)=\prod_{\substack{p\mid \ell \\ p\ne 2}}(f_{1}(1)+f_{1}(p))=\prod_{\substack{p\mid \ell \\ (p,2)=1}}\left(1+\frac{(-1)^j p}{(p-1)^{j}-(-1)^j} \right)=\frac{\mathfrak{S}_{j}(\ell)}{\mathfrak{S}^{(j)}}.
\]
Since $f_{1}(d)d^{-1}$ is a multiplicative function, we can apply Euler's product formula and, thanks to $\mu$'s properties, we have
\[
\sum_{d\ge 1}\frac{f_{1}(d)}{d}=\prod_{p\mid\ell}\left(1+\frac{f_{1}(p)}{p}\right)=\prod_{(p,2)=1}\left(1+\frac{(-1)^j}{(p-1)^{j}-(-1)^j} \right).
\]
To conclude, we have to show that
\[
\prod_{ (p,2)=1}\left(1+\frac{(-1)^j}{(p-1)^{j}-(-1)^j} \right)=\frac{2}{\mathfrak{S}^{(j)}},
\]
which follows by a simple computation. Summing up,
\begin{align*}
\sum_{\ell\le N}\frac{\mathfrak{S}_{j}(\ell)}{\mathfrak{S}^{(j)}}=\sum_{\ell\le N}\sum_{d\mid \ell} f_{1}(d)=N\sum_{d\ge 1}\frac{f_{1}(d)}{d}+ \mathcal{O}(\log N) =\frac{2N}{\mathfrak{S}^{(j)}} +\mathcal{O}(\log N),
\end{align*}
which corresponds to (17.11) of Lemma 17.4 of \cite{mon71}, for $r=1$ and $j=2$.
\end{proof}
\begin{remark}
Throughout the paper, we will adopt the notation $C':=2$, for the constant in the previous Lemma \ref{sec4:2}, since we have to add and subtract it several times.     
\end{remark}
\noindent The result below extends Lemma 2 of \cite{gol90} and Proposition of \cite{lan98}.
\begin{lemma}
    \label{sec4:3}
    For every $j\ge 3$, we have
\[
\frac{1}{N}\sum_{\ell\le N}\mathfrak{S}_{j}(\ell)^{2}=C''+\mathcal{O}\left(\frac{1}{N}\right),
\]
where
\[
C'':=2\prod_{p>2}\left( 1+\frac{1}{(p-1)^{2j-1}}\right).
\]
\end{lemma}

\begin{proof}
Recalling \eqref{eq:sigmaj} and \eqref{f_r}, for $\ell\equiv_{2}j$, we get
\begin{align*}
\mathfrak{S}_{j}(\ell)^2 =(\mathfrak{S}^{(j)})^2 \prod_{\substack{p\mid \ell \\ p>2}}\left(1+f_{1}(p)\right)^2 = (\mathfrak{S}^{(j)})^2 \prod_{\substack{p\mid \ell \\ p>2}} \left(1+g_{j}(p) \right),
\end{align*}
with
\begin{equation}
 \label{g-p}   
 g_{j}(p):=2f_{1}(p)+f_{1}(p)^2 =\frac{2(-1)^j p((p-1)^j -(-1)^j)+p^2}{((p-1)^j -(-1)^j)^2}
\end{equation}
say, for odd  $p$.
Thus
\begin{align*}
\sum_{\ell\le N} \mathfrak{S}_{j}(\ell)^2 &=(\mathfrak{S}^{(j)})^2 \sum_{2\ell\le N} \prod_{\substack{p\mid \ell \\ p>2}} (1+g_{j}(p)) 
=(\mathfrak{S}^{(j)})^2 \sum_{\ell\le N/2}\sum_{m\mid \ell} f(m),
\end{align*}
where we define 
\begin{align*}
f(m):=\begin{cases}
\mu^2 (m)\prod_{p\mid m}g_{j}(p) & \text{if $2\nmid m$} \\
0 & \text{if $2\mid m$}.
\end{cases}    
\end{align*}
Therefore, we can decompose the quadratic average as
\begin{align}
\notag
\sum_{\ell\le N} \mathfrak{S}_{j}(\ell)^2&=(\mathfrak{S}^{(j)})^{2}\sum_{m\le N/2} f(m)\biggl[\frac{N}{2m}\biggr]=\frac{(\mathfrak{S}^{(j)})^{2}}{2}\cdot N\sum_{m\ge 1}\frac{f(m)}{m} -\frac{(\mathfrak{S}^{(j)})^2}{2}\cdot N\sum_{m>N/2}\frac{f(m)}{m} \\ \label{sec4:4} &\qquad- \frac{(\mathfrak{S}^{(j)})^2}{2}\sum_{m\le N/2}f(m) -(\mathfrak{S}^{(j)})^{2}\sum_{m\le N/2} f(m)P\left( \frac{N}{2m}\right),
\end{align}
for $P(u):=u-[u]-1/2$.
The main term is given by the first summand, as
\begin{align*}
\frac{(\mathfrak{S}^{(j)})^2}{2}N\sum_{m\ge 1}\frac{f(m)}{m}&=2N\prod_{p>2}\left( 1-\frac{(-1)^j}{(p-1)^j}\right)^2 \left(1+\frac{g_{j}(p)}{p}\right) \\ &=2N\prod_{p>2} \frac{[(p-1)^{j}-(-1)^j][(p-1)^{j}-(-1)^j +2(-1)^j]+p}{(p-1)^{2j}} \\ 
&=2N\prod_{p>2}\frac{(p-1)^{2j}+(p-1)}{(p-1)^{2j}}=2N\prod_{p>2}\left(1+\frac{1}{(p-1)^{2j-1}} \right).
\end{align*}
For the third term, first of all we observe that $g_{j}(p)$ can be rewritten as
\[
g_{j}(p)=\frac{2p(-1)^j}{(p-1)^j -(-1)^j} + \frac{p^2}{((p-1)^j -(-1)^j)^2},
\]
which trivially implies that
\begin{align*}
|g_{j}(p)|\le \frac{2p}{(p-1)^j -1}+ \frac{p^2}{((p-1)^j -1)^2}  
\le \frac{3p}{(p-1)^j -1}
\end{align*}
and thus that
\begin{equation}
\label{ineq-g}
|g_{j}(p)\cdot p^{j-1}|\le \frac{3p^j}{(p-1)^j -1}.
\end{equation}
Now we observe that the right hand side of \eqref{ineq-g} is bounded by a suitable constant $B(j)>2$. In fact, if we set $x:=p-1$ and we introduce the function
\[
\delta(x)=\frac{(x+1)^j}{x^j -1},
\]
for $x\ge 2$, we get that $\delta$ is decreasing, since $\delta'(x)<0$.
Hence, we can choose $B(j)=\delta(2)$, so that $|g_{j}(p)|\le B(j)/p^{j-1}$, for all $p\ge 3$. Therefore we obtain 
\begin{align}
\notag
|f(m)|&\le \mu^2 (m)\prod_{p\mid m}|g_{j}(p)|\le \mu^2 (m)\prod_{p\mid m} \frac{B(j)}{p^{j-1}}\le\frac{B(j)^{\omega(m)}}{m^{j-1}} \\ \label{est-g} &=\frac{(2^{\omega(m)})^{\log_{2}(B(j))}}{m^{j-1}}\le\frac{d(m)^{\log_{2}(B(j))}}{m^{j-1}}.
\end{align}
Now, by Theorem 5.3 of \cite{keng82}, we know that, for a non-negative integer $A'$,
\begin{align}
\label{th-5.3}
\sum_{m\le U}d(m)^{A'} \ll U(\log U)^{2^{A'} -1}\,, \qquad \sum_{m\le U}\frac{d(m)^{A'}}{m}\ll (\log U)^{2^{A'}}. 
\end{align}
Thus, by partial summation and \eqref{th-5.3}, with $U:=N/2$ and $A'=\lceil\log_{2}(B(j)) \rceil$ we have
\begin{align*}
\sum_{m\le U} \frac{f(m)}{m^{j-1}}&\ll \sum_{m\le U}\frac{d(m)^{A'}}{m^{2}} =\frac{1}{U^2}\sum_{m\le U}d(m)^{A'} + 2\int_{1}^{U}\biggl(\sum_{m\le t}d(m)^{A'}\biggr)\frac{\dx t}{t^3}  \\
&\ll \frac{(\log U)^{2^{A'} -1}}{U} +\int_{1}^{U}\frac{(\log t)^{2^{A'} -1}}{t^2}\,\dx t \ll  1.
\end{align*}
\noindent Consequently, we also achieve the upper bound for the second term of \eqref{sec4:4}, as
\begin{align*}
\sum_{m\ge U} \frac{f(m)}{m} \ll \sum_{t\ge 0} \frac{1}{2^tU}\sum_{2^tU\le m<2^{t+1}U} f(m)\ll \frac{1}{U}.
\end{align*}
For the last term of \eqref{sec4:4}, the previous strategy for the second and the third summands works; in other terms, it holds that
\begin{equation}
\label{th-gol-fri}
\sum_{m\le U}f(m)P\left(\frac{U}{m}\right) \ll \sum_{m\le U}\frac{d(m)^{A'}}{m^{2}},
\end{equation}
for the non-negative integer $A'=\lceil \log_{2}(B(j))\rceil$ and, by partial summation, we obtain that
\[
\sum_{m\le U}f(m)P\left(\frac{U}{m}\right) \ll 1.
\]
\end{proof}



\bigskip

  AM: Dipartimento di Matematica e Informatica,
  Universit\`a degli Studi di Ferrara, Ferrara, Italy.
  email: \texttt{alessand.migliaccio@edu.unife.it}.
  AZ: [Corresponding Author]
  Dipartimento di Scienze Matematiche, Fisiche e Informatiche,
  Universit\`a di Parma, Parco Area delle Scienze, 53/a, 43124 Parma, Italy
  email: \texttt{alessandro.zaccagnini@unipr.it}

\end{document}